\documentclass[12pt]{amsart}
\usepackage{amsmath}
\usepackage{amstext}
\usepackage{amssymb}
\usepackage{amsthm}

\swapnumbers
\theoremstyle{plain}
\newtheorem{thm}{Theorem}[section]
\newtheorem{lem}[thm]{Lemma}
\newtheorem{prop}[thm]{Proposition}
\newtheorem{cor}[thm]{Corollary}

\theoremstyle{definition}
\newtheorem{defi}[thm]{Definition}

\newtheorem{ex}[thm]{Example}
\newtheorem{exs}[thm]{Examples}
\newtheorem{qu}[thm]{Question}
\newtheorem{ntn}[thm]{Notation}

\theoremstyle{remark}
\newtheorem*{note}{Note}
\newtheorem{rmk}[thm]{Remark}

\DeclareMathOperator{\Tor}{Tor}
\DeclareMathOperator{\HH}{H}
\DeclareMathOperator{\Image}{Im}
\DeclareMathOperator{\pd}{pd}

\def\fm{{\mathfrak{m}}}
\def\Q{{\mathbb{Q}}}

\begin{document}

\title{A note on freeness}
\thanks{}


\date{\today}


\begin{abstract}
In this brief note we show that for a field extension $K/F$,
$S=K[\![\mathbf{x}]\!]$ is a free $R=F[\![\mathbf{x}]\!]$-module precisely when $K/F$ is finite.
We then raise the question \emph{what is the projective dimension of $S$?}
\end{abstract}

\maketitle

\setcounter{section}{-1}
\section{\bf Introduction}
\label{in}

Throughout this paper let $K$ be a field extension of $F$, let $\mathbf{x}=x_1, \dots, x_n$ be a sequence of variables,
let $R=F[\![\mathbf{x}]\!]$, and let $S=K[\![\mathbf{x}]\!]$. We let $\fm$ denote the maximal ideal $\mathbf{x}R$.

In this note we prove that $S$ is $R$-free if and only if $K$ is a finite extension of $F$.

Our motivation originates from the study of regular rings of prime characteristic where a fundamental theorem by Ernst Kunz (\cite[Theorem 2.1]{Kunz1969})
states that a local ring $A$ of prime characteristic $p$ is regular if and only if it is reduced and $A$ is a flat $A^p$-module. The proof of that theorem
reduced to proving that if, with our notation, $F=K^p$, then $S$ is a flat $R$-module. When $K$ is a finite extension of $F=K^p$ (the \emph{$F$-finite} case) it is straightforward to show
that $S$ is $R$-free, (and hence $R$-flat), but in general it was not known whether this is the case in general.
We settle this:
a corollary of our main theorem is that $K[\![\mathbf{x}]\!]$ is free over $(K[\![\mathbf{x}]\!])^p$
precisely in the $F$-finite case.

\section{When is $S$ a free $R$-module?}

For any $f\in R$ and $d\geq 0$ define $\delta_d(f)$ to be the summand of degree $d$ in $f$.
Throughout this section $\Lambda$ will denote an index set.
We say that a sequence
$\left( f_\lambda \right)_{\lambda\in\Lambda}$ of elements in $R$
has \emph{degree-wise finite support} if
for all $d\geq 0$,
$\{ \lambda \in \Lambda : \delta_d(f_\lambda) \neq 0 \}$
is finite.

\begin{lem}
Let $T$ be a $R$-free module with free basis
$\{ e_\lambda \}_{\lambda\in \Lambda}$.
\begin{enumerate}

\item[(a)] The completion $\widehat{T}$ of $T$ consists of
the elements of $\prod_{\lambda\in\Lambda} R e_\lambda$
of degree-wise finite support, and

\item[(b)] $\widehat{T}=T$ if and only if $\Lambda$ is finite.
\end{enumerate}
\end{lem}

\begin{proof}
The elements of
\[
\widehat{T}=\lim_{\substack{\longleftarrow\\d}}\frac{ T}{\fm^d T}\cong
\lim_{\substack{\longleftarrow\\d}} \bigoplus_{\lambda\in\Lambda} \frac{R}{\fm^d} e_\lambda
\]
are precisely those $(f_\lambda)_{\lambda\in\Lambda}\in \prod_{\lambda\in\Lambda} R e_\lambda$ 
for which, for any $d\geq 0$,
the set $\{ \lambda\in \Lambda : f_\lambda\notin \fm^d \}$ is finite.

If $\Lambda$ is not finite, let $\iota: \mathbb{N} \to \Lambda$ be an injection
and let
\[
f_\lambda=
\left\{
\begin{array}{ll}
x_1^n e_\lambda & \text{if } \lambda=\iota(n) \text{ for some } n\in\mathbb{N}\\
0 & \text{otherwise}.
\end{array}
\right.
\]
Now $f$ has degree-wise finite support but $\{ \lambda\in \Lambda : f_\lambda\neq 0\}$ is not finite, so $f\notin T$.
\end{proof}

\begin{thm}\label{main theorem}
\noindent
\begin{enumerate}

\item[(a)] If $S$ is $R$-free, it has finite rank.

\item[(b)] $S$ is $R$-free if and only if $K$ is a finite extension of $F$.
\end{enumerate}

\end{thm}

\begin{proof}
Since $S$ is complete, if it is $R$-free the previous Lemma implies that it has finite rank.

If $S$ is $R$-free, say $S\cong R^k$, tensoring with $R/\fm$ yields $K\cong F^k$ and so $[K:F]=k$.

If $B$ is a finite $F$-basis for $K$, it is also a free basis for $S$ as an $R$-module.
\end{proof}

\begin{rmk}
Note that $S$ is always $R$-flat: this follows from properties of general filtered modules, see
\cite[Chapter III, \S 5]{Bourbaki}.
\end{rmk}

\begin{cor}\label{Corollary 1.4}
Let $K$ be a field of prime characteristic $p$;
$K[\![\mathbf{x}]\!]$ is a free module over $(K[\![\mathbf{x}]\!])^p$
if and only if $[K : K^p]< \infty$.
\end{cor}

\begin{proof}
Since  $K^p[\![\mathbf{x}]\!]$ is a free $(K[\![\mathbf{x}]\!])^p$-module
(with free basis $\{ x_1^{i_1} \dots x_n^{i_n} : 0\leq i_1, \dots, i_n <p\}$)
it is enough to show that $K[\![\mathbf{x}]\!]$ is a free module over $K^p[\![\mathbf{x}]\!]$
if and only if $[K : K^p]< \infty$, and this follows from  Theorem \ref{main theorem}.
\end{proof}

\section{The projective dimension of $S$}

Throughout this section we assume that $K/F$ is an infinite
extension.
Recall that projective modules over local rings are free
(\cite[Theorem 2.5]{Matsumura}), and so our assumption implies $\pd_R S>0$.

\begin{qu}
What is the projective dimension of $S$?
\end{qu}

Work by Barbara L.~Osofsky and others
on related problems (e.g., \cite{Osofsky})
suggests that answering  this question in general might
involve
going down deep set-theoretical rabbit holes.
As an example of this we prove the following.

\begin{thm}
Assume that $K$ is countable and that
$2^{\aleph_0}$ is the $m$th infinite cardinal $\aleph_m$.
Then $\pd_R S\leq m+1$.
\end{thm}
\begin{proof}
Since $S$ is flat over $R$,
\cite[Th\'{e}or\`{e}me 1.2]{Lazard}
implies that $S$ is a direct limit
of free $R$-submodules of
finite rank.
The cardinality of the set of free $R$-submodules of
finite rank is bounded by the cardinality of the
set of finite subsets of $S$, which equals the cardinality of
$S$, namely $\aleph_0^{\aleph_0}=2^{\aleph_0}=\aleph_m$.
Now the result follows from \cite[Theorem 2.44]{Osofsky}.
\end{proof}

\begin{ex}\label{Example 2.3}
As an example of a set-theoretical rabbit-hole note that
one can now conclude that
\[
 \pd_{\Q[\![\mathbf{x}]\!]} \Q(t)[\![\mathbf{x}]\!]\in\{1,2\},
\]
if we assume the Continuum Hypothesis!
\end{ex}

\section*{Acknowledgements}

I thank Mohsen Asgharzadeh for spotting a problem in the original version of Example \ref{Example 2.3}. I also thank an anonymous referee for referring me to an exercise in \cite{Bourbaki} which implies Corollary \ref{Corollary 1.4}.

\bibliographystyle{amsplain}

\end{document}